\documentclass{article}

\usepackage{graphicx, caption,subcaption}
\usepackage{amsmath, amssymb, amsthm, mathtools, dsfont}
\usepackage{color, paralist}
\usepackage[text={6.5in,9in},centering]{geometry}
\usepackage[numbers, comma, square, sort&compress]{natbib}
\usepackage{subcaption}

\newtheorem{lemma}{Lemma}
\newtheorem{theorem}{Theorem}
\newtheorem{hypothesis}{Hypothesis}

\makeatletter\@addtoreset{equation}{section}\makeatother

\newcommand{\Rg}{\mathrm{Rg}}
\newcommand{\Ns}{\mathrm{N}}
\newcommand{\rmO}{\mathrm{O}}
\newcommand{\rmd}{\mathrm{d}}
\newcommand{\rme}{\mathrm{e}}

\newcommand{\transv}{\mathrel{\text{\tpitchfork}}}
\makeatletter

\newcommand{\tpitchfork}{\vbox{\baselineskip\z@skip
\lineskip-.52ex
\lineskiplimit\maxdimen
\m@th
\ialign{##\crcr\hidewidth\smash{$-$}\hidewidth\crcr$\pitchfork$\crcr}}}
\makeatother

\begin{document}

\title{Most probable escape paths in perturbed gradient systems}

\author{%
Katherine Slyman\thanks{Division of Applied Mathematics, Brown University, Providence, RI~02912, USA}
\and Mackenzie Simper\thanks{School of Medicine, Washington University in St.\ Louis, St.\ Louis, MO~63130, USA}
\and John A. Gemmer\thanks{Department of Mathematics, Wake Forest University, Winston-Salem, NC~27109, USA}
\and Bj\"orn Sandstede\footnotemark[1]
}

\date{\today}
\maketitle



\section{Introduction}

Many natural processes can be modeled by deterministic differential equations that are subjected to small stochastic noise, in which this small noise can cause large deviations from the deterministic flow; these include climate, biological, ecological, and epidemiological systems \cite{forgoston_primer_2018}. The goal of understanding these phenomena demonstrate the demand for developing methods that quantify the impacts of noise in complex systems. A prototype model is given by
\begin{equation}\label{i:sde}
\rmd x = f(x)\, \rmd t + \epsilon\, \rmd W_t,
\end{equation}
where $x\in\mathbb{R}^n$ describes the state of the system, the smooth nonlinearity $f\colon\mathbb{R}^n\rightarrow\mathbb{R}^n$ drives the deterministic dynamics, and $W_t$ is a standard $n$-dimensional Brownian motion that reflects stochastic noise with amplitude $0<\epsilon\ll1$. Without noise, trajectories of the deterministic system
\begin{equation}\label{i:ode}
\dot{x} = f(x)
\end{equation}
will typically approach stable attractors. We are especially interested in the scenario where stable attractors consist of a finite set of asymptotically stable equilibria. In this situation, it is known that small noise can lead to large deviations from the deterministic dynamics: for instance, a solution that lies in the domain of attraction $\mathcal{A}$ of the stable equilibrium $a$ may no longer converge or stay close to the equilibrium $a$ and instead approach a different stable equilibrium. Starting with solutions $X^\epsilon(t)$ of \eqref{i:sde} with $X^\epsilon(0)=a$, we are then interested in the expected exit time $\tau^\epsilon$ at which $X^\epsilon(t)$ leaves the domain of attraction $\mathcal{A}$, the exit location $X^\epsilon(\tau^\epsilon)\in\partial\mathcal{A}$ through which the solution leaves $\mathcal{A}$, and the optimal (that is, most probable) escape paths that $X^\epsilon(t)$ will follow in $\mathcal{A}$ for $0\leq t\leq\tau^\epsilon$. These questions have been investigated thoroughly in many works, and we refer to \cite{freidlin_random_2012, berglund_kramers_2013, forgoston_primer_2018} for results and references.

Our focus will be on optimal escape paths. These can be found as minimizers of appropriate functionals on bounded or unbounded time intervals or, equivalently, as solutions to the corresponding Euler--Lagrange equations. 
More concrete information about the optimal paths is available if the deterministic system \eqref{i:ode} is a gradient system so that $f(x)=-\nabla V(x)$ for an appropriate potential $V\colon\mathbb{R}^n\to\mathbb{R}$. In this case, trajectories $X^\epsilon(t)$ of \eqref{i:sde} will typically leave the deterministic domain of attraction $\mathcal{A}$ of a stable equilibrium $a$ of \eqref{i:ode} through the saddle equilibrium $b\in\partial\mathcal{A}$ with $b=\arg\min\{V(z)\colon z\in\partial\mathcal{A}\}$. Moreover, $X^\epsilon(t)$ will follow the time-reversal of the heteroclinic orbit $h(t)$ of \eqref{i:ode} that connects $x=b$ at $t=-\infty$ with $x=a$ at $t=\infty$ so that $X^\epsilon(t)\simeq h(-t)$ \cite{freidlin_random_2012, berglund_kramers_2013, forgoston_primer_2018}.

For non-gradient systems, it is expected that the optimal escape path will be different from the time-reversed connecting orbit $h(-t)$. In this paper we study the case
\[
\dot{x} = -\nabla V(x) + \mu g(x),
\]
where $\mu \in \mathbb{R}$ allows for a perturbation to a gradient system, and prove the following statements.
\begin{compactenum}[(i)]
\item For symmetric perturbations $g(x)$ for which $g_x(x)=g_x(x)^*$ for all $x$, optimal escape paths are still given by the time-reversals of heteroclinic orbits connecting saddles and attractors.
\item For non-symmetric perturbations $g(x)$ (and, more precisely, when $[g_x(h(t))-g_x(h(t))^*]\dot{h}(t)$ does not vanish identically), the time-reversed heteroclinic saddle-attractor connection and the optimal escape path will differ at order $\mu$.
\end{compactenum}
The remainder of this paper is organized as follows. In \S\ref{s:2}, we briefly review the relevant large-deviations results for exit problems. Section~\ref{s:3} contains a discussion of the Euler--Lagrange equations for symmetric deterministic nonlinearities, and we prove our main result on optimal escape paths for non-gradient perturbations in \S\ref{s:4}. We illustrate our results with an example system in \S\ref{s:5} and finish with a discussion in~\S\ref{s:6}.


\section{Review: Large deviations for the exit problem}\label{s:2}

We review results on large deviations for a deterministic dynamical system $\dot{x}=f(x)$ perturbed by small noise. Consider the stochastic differential equation
\begin{equation} \label{e:sde}
\rmd x = f(x)\,\rmd t + \epsilon\,\rmd W_t,
\end{equation}
where $x\in\mathbb{R}^n$, $f\colon\mathbb{R}^n\rightarrow\mathbb{R}^n$ is smooth, $W_t$ is a standard $n$-dimensional Brownian motion, and $0<\epsilon\ll1$. We assume that $x=a$ is an asymptotically stable equilibrium of the deterministic system
\begin{equation} \label{e:ode0}
\dot{x} = f(x)
\end{equation}
with domain of attraction $\mathcal{A}:=\{x(0)\colon x(t)\mbox{ satisfies \eqref{e:ode0}, }\lim_{t\rightarrow\infty}x(t)=a\}$. We are interested in trajectories of \eqref{e:sde} that start at $x=a$ and leave $\mathcal{A}$ in finite time for $\epsilon>0$. To describe these solutions, we define the function space
\[
\Gamma_{[T_0,T_1]} := \left\{ u \in C^0([T_0,T_1],\mathbb{R}^n)\colon u(T_0)=a,\, u(t)\in\mathcal{A} \mbox{ for } T_0\leq t\leq T_1 \mbox{, and } u \mbox{ is absolutely continuous} \right\},
\]
the functional
\begin{equation}
\mathcal{I}\colon\quad
\Gamma_{T_0,T_1]} \longrightarrow\mathbb{R},\quad
u\longmapsto \mathcal{I}(u) := \frac{1}{2} \int_{T_0}^{T_1} |\dot{u}(s)-f(u(s))|^2\,\rmd s
\label{e:rate functional}
\end{equation}
on $\Gamma_{[T_0,T_1]}$, and the associated quasi-potential
\begin{equation}\label{e:V}
\mathcal{V}(z) := \inf\left\{ \mathcal{I}(u)\colon u\in\Gamma_{T_0,T_1]},\, u(T_1)=z,\, T_0<T_1 \right\}
\end{equation}
defined for $z\in\overline{\mathcal{A}}$. We set $\mathcal{V}_{\partial\mathcal{A}}:=\inf\{\mathcal{V}(z)\colon z\in\partial\mathcal{A}\}>0$ and consider solutions $X^\epsilon(t)$ with $X^\epsilon(0)=a$ of \eqref{e:sde}. The following statements are then true for these solutions under appropriate regularity assumptions on $\mathcal{A}$:
\begin{compactenum}
\item[\textbf{Exit time:}]
Let $\tau^\epsilon_\mathcal{A}=\inf\{t>0\colon X^\epsilon(t)\in\partial\mathcal{A}\}$ be the exit time from $\mathcal{A}$, then the expectation of the exit time satisfies $\lim_{\epsilon\to0} \epsilon \log E(\tau^\epsilon_\mathcal{A}) = \mathcal{V}_{\partial\mathcal{A}}$ (see \cite[Theorem~4.2]{Day}). In particular, the expected exit time is finite for finite $\epsilon$ and approaches $\infty$ as $\epsilon\to0$.
\item[\textbf{Exit location:}]
If $b:=\arg\min_{z\in\partial\mathcal{A}}\mathcal{V}(z)$ is unique, then $\lim_{\epsilon\to0}P(\{|X^\epsilon(\tau^\varepsilon_\mathcal{A})-b|<\delta\})=1$ for each $\delta>0$ (see \cite[Theorem~3.1 and \S5]{Day}). In particular, the most probable exit location on the boundary of the domain of attraction is the minimizer of the quasi-potential on the boundary.
\item[\textbf{Exit path:}]
Assume $b:=\arg\min_{z\in\partial\mathcal{A}}\mathcal{V}(z)$ is unique, then for each $\eta>0$ we have
\begin{equation}\label{st:ep}
\lim_{\epsilon\to0} P\left(\left\{ \max_{\tau^\epsilon_\eta\leq t\leq\tau^\epsilon_\mathcal{A}} |X^\epsilon(t)-u_*(t-\tau^\epsilon_\eta+\tau_\eta)| < \delta \right\}\right) = 1
\end{equation}
for each $\delta>0$, where $\tau^\epsilon_\eta$ and $\tau_\eta$ are the largest times with $|X^\epsilon(\tau^\epsilon_\eta)-a|=\eta$ and $|u_*(\tau_\eta)-a|=\eta$, respectively, and $u_*$ minimizes $\mathcal{I}$ on $(-\infty,T]$ with $u_*(-\infty)=a$ and $u_*(T)\in\partial\mathcal{A}$ (see \cite[Theorem~2.3 in Chapter~4]{freidlin_random_2012} and \cite[Lemma~3.3]{WentzellFreidlin1970}). In particular, the most probable exit path through the boundary is given by the minimizer of the functional $\mathcal{I}$ on $(-\infty,T]$.
\end{compactenum}
We can characterize the most probable exit path $u_*$ as a solution to the Euler-Lagrange equation associated with the functional $\mathcal{I}(u)$. Assume that $u_*$ is a minimizer of the functional, then taking the Gateaux derivative of the functional at $u_*$ and setting the derivative to zero leads to the Euler--Lagrange equation
\begin{equation}\label{e:EL0}
\ddot{u} = f_u(u) \dot{u} - f_{u}(u)^*(\dot{u}-f(u))
\end{equation}
that $u=u_*$ must satisfy. Following \cite{forgoston_primer_2018}, we use the Legendre transform $w=\dot{u}-f(u)$ to obtain the equivalent first-order Hamiltonian system
\begin{align}\label{e:EL1}
\dot{u} & = f(u) + w \\ \nonumber
\dot{w} & = - f_{u}(u)^* w
\end{align}
with corresponding Hamiltonian
\begin{equation}\label{e:HO}
H(u,w) = \frac{|w|^2}{2} + \langle f(u),w \rangle. 
\end{equation}

The theory of large deviations yields concrete information when the nonlinearity $f$ is the gradient of a smooth potential $V\colon\mathbb{R}^n\rightarrow\mathbb{R}$ so that $f(x)=-\nabla V(x)$. In this case, if $u\in\Gamma_{T_0,T_1]}$ satisfies $u(T_1)=b\in\partial\mathcal{A}$, then we have
\begin{align*}
\mathcal{I}(u)
& = \frac{1}{2} \int_{T_0}^{T_1} |\dot{u}(t) + \nabla V(u(t))|^2\,\rmd t 
  = \frac{1}{2} \int_{T_0}^{T_1} |\dot{u}(t) - \nabla V(u(t))|^2\,\rmd t
+ 2 \int_{T_0}^{T_1} \left\langle\dot{u}(t),\nabla V(u(t))\right\rangle\,\rmd t \\
& = \frac{1}{2} \int_{T_0}^{T_1} |\dot{u}(t)-\nabla V(u(t))|^2\,\rmd t + 2(V(b)-V(a))
 \geq 2(V(b)-V(a)).
\end{align*}
The lower bound is attained if and only if $\dot{u}(t)=\nabla V(u(t))=-f(u(t))$, that is, if $u(t)$ is a solution of the time-reversed deterministic dynamical system $\dot{y}=-f(y)$ with $u(T_0)=a$ and $u(T_1)=b$. This result can be extended to the case $[T_0,T_1]=\mathbb{R}$ to establish the following lemma.

\begin{lemma}[{\cite[Theorem~3.1 in Chapter~4]{freidlin_random_2012}}]\label{l:grad}
Assume that $f(x)=-\nabla V(x)$ is a gradient and that $b=\arg\min_{z\in\partial\mathcal{A}} V(z)$ is unique and a saddle equilibrium of $\dot{x}=f(x)$. Then there is a unique heteroclinic orbit $h(t)\in W^\mathrm{u}(b)\cap W^\mathrm{s}(a)$ of $\dot{x}=f(x)$ connecting $x=b$ and $x=a$, and the function $u_*(t):=h(-t)$ satisfies $u_*=\arg\min\{\mathcal{I}(u)\colon u\in\Gamma_\mathbb{R}\}$.
\end{lemma}

In particular, the most probable exit path is given by the time-reversal of the heteroclinic orbit of \eqref{e:ode0} that connect the saddle $b$ to the attractor $a$.


\section{Euler--Lagrange dynamics for symmetric vector fields}\label{s:3}

Consider the Euler--Lagrange equations \eqref{e:EL1} given by
\begin{align}\label{e:EL2}
\dot{u} & = f(u) + w \\ \nonumber
\dot{w} & = -f_u(u)^* w.
\end{align}
Equilibria of \eqref{e:EL2} are of the form $(u,w)=(e,0)$ with $f(e)=0$ or $(u,w)=(e,-f(e))$ with $f_u(e)^*f(e)=0$. We are interested in the equilibria of the form $(u,w)=(e,0)$, which are the equilibria $x=e$ of the deterministic system
\begin{equation}\label{e:ode}
\dot{x} = f(x), \qquad x\in\mathbb{R}^n.
\end{equation}
If $x=e$ is a hyperbolic equilibrium of \eqref{e:ode}, then $(u,w)=(e,0)$ is a hyperbolic equilibrium of \eqref{e:EL2} with $n$ stable and $n$ unstable eigenvalues. Note also that the set $w=0$ is invariant under \eqref{e:EL2} and that \eqref{e:EL2} restricted to $w=0$ becomes the deterministic system \eqref{e:ode}.

We introduce the new variable $v=w+2f(u)$ for which \eqref{e:EL2} becomes
\begin{align}\label{e:EL}
\dot{u} & = -f(u) + v \\ \nonumber
\dot{v} & = 2(f_u(u)^*-f_u(u)) f(u) + (2f_u(u)-f_u(u)^*) v
\end{align}
with conserved quantity
\begin{equation}\label{e:ELC}
C(u,v) = \frac12|v|^2-\langle f(u),v\rangle.
\end{equation}
We first focus on nonlinearities $f$ with a symmetric Jacobian.

\begin{hypothesis}\label{h1}
The nonlinearity $f(x)$ satisfies $f_x(x)=f_x(x)^*$ for all $x$.
\end{hypothesis}

Note that gradient vector fields of the form $f(x)=-\nabla V(x)$ automatically satisfy \ref{h1}, since the Hessian of the potential $V$ is always symmetric. From now on, we assume that \ref{h1} is met so that \eqref{e:EL} reduces to
\begin{align}\label{e:ELs}
\dot{u} & = -f(u) + v \\ \nonumber
\dot{v} & = f_u(u)^* v.
\end{align}
Note that \eqref{e:ELs} admits solutions of the form $(u,v)(t)=(y(t),0)$ for each solution $y(t)$ of the time-reversed differential equation
\begin{equation}\label{e:-ode}
\dot{y} = -f(y).
\end{equation}
In other words, each solution $x(t)$ of the deterministic system \eqref{e:ode} provides a solution $(u,v)(t):=(x(-t),0)$ of \eqref{e:ELs}. Equation \eqref{e:ELs} has the Hamiltonian $H(u,v)$ defined by
\begin{equation}\label{e:H}
H(u,v) = \frac{|v|^2}{2} - \langle f(u), v\rangle, \qquad
\nabla H(u,v) = \begin{pmatrix} -f_u(u) v \\ v-f(u) \end{pmatrix}.
\end{equation}
We will focus on transverse heteroclinic orbits $x(t)=h(t)$ of \eqref{e:ode} that connect a saddle $x=b$ at $t=-\infty$ to an attractor $x=a$ at $t=\infty$ as encoded in the following hypothesis.

\begin{hypothesis}\label{h2}
Assume that $x=a$ and $x=b$ are two hyperbolic equilibria of \eqref{e:ode} with $\dim W^\mathrm{s}(a)=n$ and $\dim W^\mathrm{u}(b)=1$, and that $h(t)$ is a solution of $\dot{x}=f(x)$ with $h(0)\in W^\mathrm{u}(b)\transv W^\mathrm{s}(a)$.
\end{hypothesis}

Our first result shows that the associated solution $(u_*,v_*)(t)=(h(-t),0)$ is then a heteroclinic orbit of \eqref{e:ELs} that connects the saddle $(a,0)$ to the saddle $(b,0)$ and is transverse inside the Hamiltonian level set $H^{-1}(0)$.

\begin{lemma}\label{l:trans}
Assume that Hypotheses \ref{h1} and \ref{h2} are met, then $(u_1,v_1)=(a,0)$ and $(u_2,v_2)=(b,0)$ are hyperbolic saddle equilibria of \eqref{e:ELs} with unstable and stable dimension equal to $n$, and the solution $(u_*,v_*)(t):=(h(-t),0)$ of \eqref{e:ELs} satisfies $(u_*,v_*)(0)\in W^\mathrm{u}(a,0)\transv W^\mathrm{s}(b,0)$ inside $H^{-1}(0)$.
\end{lemma}

\begin{proof}
The linearization of \eqref{e:ELs} around $(u_*,v_*)(t)=(h(-t),0)\in\mathbb{R}^n\times\mathbb{R}^n$ is given by
\begin{align}\label{e:ELsl}
\dot{u} & = -f_u(u_*(t)) u + v \\ \nonumber
\dot{v} & = f_u(u_*(t))^* v.
\end{align}
Using the block-diagonal structure of \eqref{e:ELsl}, we see that $(a,0)$ and $(b,0)$ are hyperbolic saddle equilibria of \eqref{e:ELs} with $\dim W^\mathrm{u}(a,0)=\dim W^\mathrm{s}(b,0)=n$, $W^\mathrm{u}(a,0)\cup\dim W^\mathrm{s}(b,0)\subset H^{-1}(0)$, and $(u_*,v_*)(0)\in W^\mathrm{u}(a,0)\cap\dim W^\mathrm{s}(b,0)$. Note that $H^{-1}(0)$ is manifold of dimension $2n-1$ along the solution $(u_*(t),0)$ since the gradient $\nabla H(u_*(t),0)=(0,-f(u_*(t)))^*$ does not vanish. To establish transversality, it therefore suffices to prove that $(u_*,v_*)(t)=(\dot{u}_*(t),0)$ is, up to constant scalar multiples, the only bounded solution of \eqref{e:ELsl}. We focus first on the second, decoupled equation $\dot{v}=f_u(u_*(t))^*v$ in \eqref{e:ELsl}: since $f_u(u_*(t))=f_x(h(-t))\rightarrow f_x(a)$ as $t\to-\infty$, and the eigenvalues of $f_x(a)$ lie in the open left half-plane, we find that $v(t)\equiv0$ is the only solution of $\dot{v}=f_u(u_*(t))^*v$ that is bounded uniformly in $t\in\mathbb{R}$. Hence, we necessarily have $v=0$ in \eqref{e:ELsl}, and it suffices to identify bounded solutions of $\dot{u}=-f_u(u_*(t))u=-f_u(h(-t))u$ or, equivalently, of $\dot{x}=f_u(h(t))x$ after reversing time. The latter equation agrees with the linearization of \eqref{e:ode} around $h(t)$, and \ref{h2} ensures that $x(t)=\dot{h}(t)$ is the only solution, up to constant scalar multiples, that is bounded uniformly in $t$. This completes the proof of the lemma.
\end{proof}


\section{Euler--Lagrange dynamics under non-symmetric perturbations}\label{s:4}

In this section, we investigate whether the most probable escape path is always given by the heteroclinic orbit $y(t)=h(-t)$ of \eqref{e:-ode} that connects the attractor $x=a$ to the saddle equilibrium $x=b$. To do so, we start with a nonlinearity $f(x)$ that satisfies Hypotheses \ref{h1}-\ref{h2} and add a small non-symmetric perturbation to understand how the heteroclinic orbit $y(t)=h(-t)$ of \eqref{e:-ode} and the most probable path $u(t)$, obtained as a heteroclinic orbit of the Euler--Lagrange equation \eqref{e:EL} change and whether they differ. Thus, consider the time-reversed system
\begin{equation}\label{e:pode}
\dot{y} = - f(y) - \mu g(y)
\end{equation}
and the Euler--Lagrange equation
\begin{align}\label{e:pEL}
\dot{u} & = - f(u) - \mu g(u) + v \\ \nonumber
\dot{v} & = f_u(u)^* v + \mu \Big[ 2(g_u(u)^*-g_u(u)) (f(u)+\mu g(u)) + (2g_u(u)-g_u(u)^*) v\Big]
\end{align}
for $\mu$ with $|\mu|\ll1$. Lemma~\ref{l:trans} and the existence of a conserved quantity of \eqref{e:pEL} based on \eqref{e:EL} and \eqref{e:ELC} imply that the transverse heteroclinic orbits $y_0(t)=h(-t)$ and $(u_0,v_0)(t)=(h(-t),0)$ of \eqref{e:-ode} and \eqref{e:ELs} persist as unique heteroclinic solutions $y_*(t;\mu)$ and $(u_*,v_*)(t;\mu)$ of \eqref{e:pode} and \eqref{e:pEL}, respectively, for all small $\mu$ and are smooth in $\mu$.

\begin{theorem}\label{t:trans}
Assume that Hypotheses \ref{h1} and \ref{h2} are met and that there is an $s\in\mathbb{R}$ so that
\begin{equation}\label{e:g}
\Big[ g_u(y_0(s))^* - g_u(y_0(s)) \Big] f(y_0(s)) \neq 0,
\end{equation}
then there is a constant $\delta>0$ so that $\sup_{t\in\mathbb{R}}|y_*(t;\mu)-u_*(t;\mu)|\geq\delta|\mu|$ for all $0<|\mu|\ll1$.
\end{theorem}

In particular, this result implies that the most probable escape path $u_*(t;\mu)$ will not be given by the heteroclinic connection $y_*(t;\mu)$ for each $\mu$ with $0<|\mu|\ll1$.

\begin{proof}
Using smoothness in $\mu$, we expand the solutions $y_*(t;\mu)$ and $(u_*,v_*)(t;\mu)$ as
\[
y_*(t;\mu) = h(-t) + \mu y_1(t) + \rmO(|\mu|^2), \qquad
(u_*,v_*)(t;\mu) = (h(-t),0) + \mu (u_1,v_1)(t) + \rmO(|\mu|^2)
\]
and note that the $\rmO(|\mu|^2)$ terms are bounded by $|\mu|^2$ uniformly in $t\in\mathbb{R}$. Using the notation $y_0(t)=h(-t)$, the leading-order contributions $y_1(t)$ and $(u_1,v_1)(t)$ are bounded solutions of the linear differential equations
\begin{equation}\label{e:1}
\dot{y}_1 = - f_u(y_0(t)) y_1 - g(y_0(t))
\end{equation}
and
\begin{align}\label{e:2}
\dot{u}_1 & = -f_u(y_0(t)) u_1 + v_1 - g(y_0(t)) \\ \nonumber
\dot{v}_1 & =  f_u(y_0(t))^* v_1 + 2(g_u(y_0(t))^*-g_u(y_0(t))) f(y_0(t)),
\end{align}
respectively. The linear equation
\begin{equation}\label{e:3}
\dot{y} = - f_u(y_0(t)) y
\end{equation}
has exponential dichotomies $\Phi_+^\mathrm{s,u}(t,s)$ for $t,s\geq0$ and $\Phi_-^\mathrm{u}(t,s)$ for $t,s\leq0$, and the unique bounded solution of
\begin{equation}\label{e:4}
\dot{y} = - f_u(y_0(t)) y + g(t), \qquad g\in C^0(\mathbb{R},\mathbb{R}^n)
\end{equation}
is therefore given by
\begin{equation}\label{e:5}
y(t) =
\begin{cases} \displaystyle
\int_0^t \Phi_+^\mathrm{s}(t,s) g(s)\,\rmd s
+ \int_\infty^t \Phi_+^\mathrm{u}(t,s) g(s)\,\rmd s & t\geq0 \\ \displaystyle
\Phi_-^\mathrm{u}(t,0) \int_\infty^0 \Phi_+^\mathrm{u}(0,s) g(s)\,\rmd s
+ \int_0^t \Phi_-^\mathrm{u}(t,s) g(s)\,\rmd s & t\leq0.
\end{cases}
\end{equation}
We conclude that
\begin{equation}\label{e:y1}
y_1(t) =
\begin{cases} \displaystyle
- \int_0^t \Phi_+^\mathrm{s}(t,s) g(y_0(s))\,\rmd s
- \int_\infty^t \Phi_+^\mathrm{u}(t,s) g(y_0(s))\,\rmd s & t\geq0 \\ \displaystyle
- \Phi_-^\mathrm{u}(t,0) \int_\infty^0 \Phi_+^\mathrm{u}(0,s) g(y_0(s))\,\rmd s
- \int_0^t \Phi_-^\mathrm{u}(t,s) g(y_0(s))\,\rmd s & t\leq0
\end{cases}
\end{equation}
is the unique bounded solution of \eqref{e:1}. We consider \eqref{e:2} next and define
\[
g_1(s) := 2\Big[ g_u(y_0(t))^*-g_u(y_0(t)) \Big] f(y_0(t))
\]
so that we can write the equation for $v_1$ as
\begin{equation}\label{e:lu=g}
\mathcal{L} v_1 = \left[\frac{\rmd}{\rmd t}-f_u(y_0(\cdot))^* \right] v_1 = g_1.
\end{equation}
It follows from \cite[Lemma~4.2]{Palmer} that $\mathcal{L}\colon C^1(\mathbb{R},\mathbb{R}^n)\rightarrow C^0(\mathbb{R},\mathbb{R}^n)$ is a bounded Fredholm operator with index $-1$, null space $\Ns(\mathcal{L})=\{0\}$, and range $\Rg(\mathcal{L})^{\perp}=\mathbb{R}f(y_0)$, and that $\mathcal{L}v_1=g_1$ has a solution if and only if $\langle f(y_0),g_1\rangle_{L^2}=0$. We use the latter condition to check solvability of \eqref{e:lu=g}: for each $t\in\mathbb{R}$, we have
\begin{align*}
\frac{1}{2}\langle f(y_0(t)),g_1(t)\rangle
& = \langle f(y_0(t)), (g_x(y_0(t))^* - g_x(y_0(t)) f(y_0(t)) \rangle \\
& = \langle g_x(y_0(t)) f(y_0(t)), f(y_0(t)) \rangle - \langle f(y_0(t)), g_x(y_0(t))f(y_0(t)) \rangle\\
& = 0.
\end{align*}
Hence, $\langle f(y_0),g_1\rangle_{L^2}=0$, and we conclude that there exists a unique $v_1\in C^1(\mathbb{R},\mathbb{R}^n)$ with $\mathcal{L}v_1=g_1$, and $g_1\not\equiv0$ implies $v_1\not\equiv0$. It remains to solve the remaining equation
\begin{equation}\label{e:eu1}
\dot{u}_1 = -f_x(y_0(t)) u_1 + v_1(t) - g(y_0(t))
\end{equation}
where $v_1=\mathcal{L}^{-1}g_1$ as above. This equation is of the form \eqref{e:4} and its unique bounded solution $u_1(t)$ is therefore of the form \eqref{e:5}. Comparing \eqref{e:1} and \eqref{e:eu1}, we see that
\begin{align}\label{e:u1}
u_1(t) & = y_1(t) + 
\begin{cases} \displaystyle
\int_0^t \Phi_+^\mathrm{s}(t,s) v_1(s)\,\rmd s
+ \int_\infty^t \Phi_+^\mathrm{u}(t,s) v_1(s)\,\rmd s & t\geq0 \\ \displaystyle
\Phi_-^\mathrm{u}(t,0) \int_\infty^0 \Phi_+^\mathrm{u}(0,s) v_1(s)\,\rmd s 
+ \int_0^t \Phi_-^\mathrm{u}(t,s) v_1(s)\,\rmd s & t\leq0
\end{cases} \\ \nonumber
& =: y_1(t) + \Delta_1(t).
\end{align}
We conclude that $u_1(t)-y_1(t)=\Delta_1(t)$, and \eqref{e:u1} shows that $\Delta_1\equiv0$ if and only if $v_1\equiv0$. We argued above that the latter holds if and only if $g_1\equiv0$. Thus, whenever there is a $t\in\mathbb{R}$ for which
\[
(g_u(y_0(t))^*-g_u(y_0(t))) f(y_0(t)) \neq0,
\]
we conclude that $u_*(t)$ will differ from the heteroclinic connection $y_*(t)$ for each $\mu$ with $0<|\mu|\ll1$. This completes the proof of the lemma.
\end{proof}

We note that \cite{Chao_SIADS} similarly used Melnikov calculations to investigate noisy transitions in a one-dimensional system perturbed by weak periodic forcing. While we do not focus on  periodic states, our calculations are more general and apply to systems in arbitrary dimensions.


\section{Numerical illustration}\label{s:5}

To illustrate how numerically computed optimal paths compare with our theoretical results, and to demonstrate Theorem~\ref{t:trans}, we consider the planar system
\begin{equation}\label{e:ode1}
\dot{x} = f(x) + \mu g(x) = -\nabla V(x) + \mu g(x), \qquad x=(x_1,x_2)\in\mathbb{R}^2,
\end{equation}
where $V$ and the non-gradient perturbation $g$ are given by
\begin{equation} \label{eqn: gradient}
\nabla V(x) = \begin{pmatrix} x_1 (x_1^2 -1) \\ x_2  \end{pmatrix}, \qquad
g(x) = \begin{pmatrix} -x_2 \\ 0 \end{pmatrix}.
\end{equation}
The system \eqref{e:ode1} has two stable equilibria at $a:=(-1,0)$ and $(1,0)$ that are separated by the saddle $b:=(0,0)$. We consider the most probable escape path from the attractor $a$. The heteroclinic orbit $h(t)$ of \eqref{e:ode1} that connects $x=b$ at $t=-\infty$ to $x=a$ at $t=\infty$ does not depend on $\mu$ and is given by
\[
h(t) = \begin{pmatrix} \frac{-1}{\sqrt{\rme^{-2t}+1}} \\ 0 \end{pmatrix}.
\]
Lemma~\ref{l:trans} implies that the corresponding solution $(u_0,v_0)(t)\in\mathbb{R}^4$ of the associated Euler--Lagrange equation at $\mu=0$ is therefore given by
\[
\begin{pmatrix} u_0(t) \\ v_0(t) \end{pmatrix} = \begin{pmatrix} h(-t) \\ 0 \end{pmatrix}.
\]
To apply Theorem~\ref{t:trans}, we need to check condition \eqref{e:g} for our example. Since the quantity 
\[
\left[ g_x(h(-t))^* - g_x(h(-t)) \right] f(h(-t)) = \begin{pmatrix} 0 \\ 2\dot{h}_1(-t) \end{pmatrix}
\]
never vanishes, it follows from Theorem~\ref{t:trans} that the most probable exit path $u_\mathrm{exit}(t;\mu)$ will differ from the time-reversed heteroclinic orbit $h(-t)$ for $0<|\mu|\ll1$. We can explicitly solve the equation \eqref{e:eu1} for $u_1(t)\in\mathbb{R}^2$ to get
\[
u_1(t) = \begin{pmatrix} 0 \\ -\rme^t + \sqrt{\rme^{2t}+1} - \rme^{-t} \sinh^{-1}(\rme^t) \end{pmatrix} ,
\]
and the most probable exit path $u_\mathrm{exit}(t;\mu)$ is therefore given by
\begin{equation}
u_\mathrm{exit}(t;\mu) = u_0(t) + \mu u_1(t) + \rmO(|\mu|^2) = 
\begin{pmatrix} \frac{-1}{\sqrt{\rme^{2t}+1}} \\ 0 \end{pmatrix}
+ \mu \begin{pmatrix} 0 \\ \sqrt{\rme^{2t}+1} - \rme^t - \rme^{-t} \sinh^{-1}(\rme^t) \end{pmatrix} 
+ \rmO(|\mu|^2).
\label{eq:ex 1 analytical}
\end{equation}
In Figure~\ref{fig:ex1}, we compare the first-order approximation $u_\mathrm{approx}(t;\mu):=u_0(t)+\mu u_1(t)$ of $u_\mathrm{exit}(t;\mu)$ with its numerical approximation $u_\mathrm{num}(t;\mu)$, which we compute by finding the heteroclinic connection of the associated Euler--Lagrange system \eqref{e:pEL} using \textsc{auto07p}. This figure demonstrates that the most probable exit path $u_\mathrm{exit}(t;\mu)$ differs from the $\mu$-independent time-reversed heteroclinic orbit $h(t)$ of \eqref{e:ode1} for $0<|\mu|\ll1$.

\begin{figure}[ht]
\centering
\begin{subfigure}[b]{0.45\textwidth}
\includegraphics[width=\textwidth]{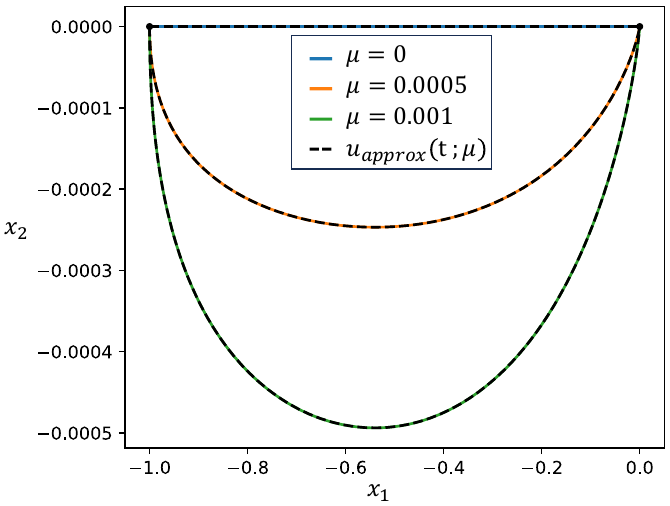}
\caption{}
\label{fig:ex1_paths}
\end{subfigure}
\hspace{5mm}
\begin{subfigure}[b]{0.45\textwidth}
\includegraphics[width=\textwidth]{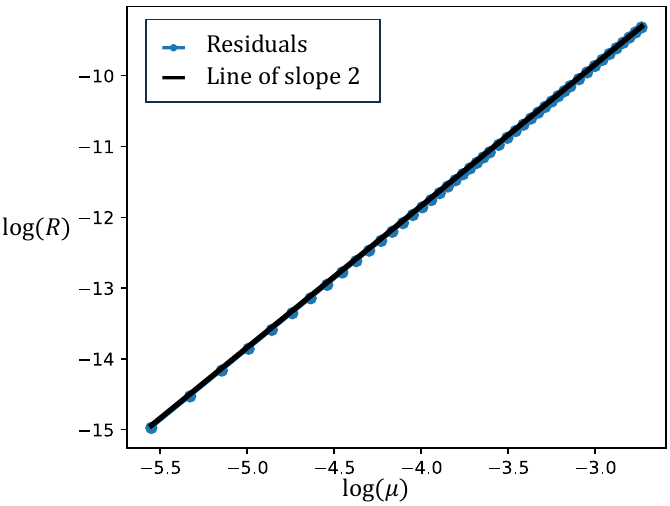}
\caption{}
\label{fig:ex1_error}
\end{subfigure}
\caption{Panel~(a): Shown are the numerical approximation $u_\mathrm{approx}(t;\mu)$  from the attractor $(-1,0)$ to the saddle $(0,0)$ of the most probable exit paths $u_\mathrm{exit}(t;\mu)$ for $\mu=0$ (blue), $\mu=0.0005$ (orange), and $\mu=0.001$ (green) together with the analytical approximations $u_\mathrm{approx}(t;\mu)$ in black (note that the $\mu$-independent time-reversed heteroclinic orbit lies on the line $x_2=0$). Panel~(b): Shown is the log-log plot of the $L^2$-difference $R:=\|u_\mathrm{approx}(\cdot;\mu)-u_\mathrm{num}(\cdot;\mu)\|_{L^2}$ of the first-order analytical approximation $u_\mathrm{approx}$ and the numerical approximation $u_\mathrm{num}$ of the exit path $u_\mathrm{exit}$ against $\mu$. As expected, the linear fit has slope $2$.}
\label{fig:ex1}
\end{figure}


\section{Discussion}\label{s:6}

The dynamics of stochastic models can differ significantly from the dynamics of an underlying deterministic system. Noise can affect the dynamics on a variety of scales and can induce rare transitions away from deterministic attractors  \cite{forgoston_primer_2018}; these are often referred to as noise-induced tipping points \cite{Ashwin12}. There are three objects relevant to noise-induced rare transitions: the probability distribution of escape points on a basin boundary, the most probable escape path from an attractor, and the expected time of this escape. For gradient systems, it is well known that the most probable escape path is given by the time-reversed heteroclinic connection between a saddle and the attractor of the deterministic system in the limit of vanishing noise amplitude. In this case, it is also known that an approximation of the expected time $\tau$ to tip from an attractor $a$ to a different attractor that is separated by a saddle $b$ is given by
\begin{equation*}
\mathbb{E}[\tau] \approx C \rme^{\frac{2(V(b)-V(a))}{\varepsilon^2}},
\label{EQ: expected time to tip formula}
\end{equation*}
where $V$ is the potential of the gradient system and $\varepsilon$ the noise amplitude. For gradient systems, the prefactor $C$ can be explicitly computed to O($\sqrt{\epsilon}$) and, in the one-dimensional case, it can be found using Laplace's method. While not the focus of this work, we believe it is possible to follow the techniques of \cite{grafke,bouchet2022path} to compute an asymptotic expansion of $C$ for the expected escape time for $\rmO(\mu)$ perturbations of a gradient system.

The focus of our work was to characterize the most probable escape paths for non-gradient perturbations of gradient systems and to understand whether, and by how much, these escape paths would differ from the time-reversed heteroclinic attractor-saddle connection. We accomplished this as follows. First, we identified the most probable escape path with an appropriate heteroclinic saddle-saddle connection of the Euler--Lagrange equation associated with the large-deviations functional \eqref{e:rate functional}. Next, we used Melnikov theory to compute the displacements of the time-reversed attractor-saddle connection of the deterministic model and of the heteroclinic saddle-saddle connection of the Euler--Lagrange equation.

Our result shows that if the Jacobian of the perturbation of the deterministic model is not symmetric along the unperturbed heteroclinic attractor-saddle connection, then the most probable escape path will differ from the deterministic attractor-saddle connection. We also showed more generally that for any system with symmetric Jacobian, whether with gradient structure or not, the most probable escape paths will coincide with the deterministic attractor-saddle connections.

One advantage of the characterization we used here is that it avoids the computation of quasi-potentials or the use of Monte--Carlo simulations. In particular, calculating the quasipotential in higher dimensions is a numerically challenging task since it involves solving the static Hamilton--Jacobi equations \cite{ballscups, CAMERON20121532}.

\section{Acknowledgements}
This material is based upon work supported by the National Science Foundation. K. Slyman was supported by the NSF under grant RTG: DMS-2038039. B. Sandstede was partially supported by the NSF through grants RTG: DMS-2038039 and DMS-2106566.


\bibliographystyle{siam}
\bibliography{BIB}

\end{document}